\documentclass[11pt,reqno]{amsart}
\usepackage{amssymb,latexsym}
\usepackage{graphicx}
\usepackage{mathtools}
\usepackage{color}
\usepackage[T1]{fontenc}
\usepackage{hyperref}
\usepackage{amsmath}
\usepackage{breqn}
\usepackage[toc,page]{appendix}
\usepackage{url}    
\usepackage{cancel}
\usepackage{multirow}
\newcommand{\bburl}[1]{\textcolor{blue}{\url{#1}}}

\calclayout

\makeatletter
\newcommand{\monthyear}[1]{%
  \def\@monthyear{\uppercase{#1}}}
\newcommand{\volnumber}[1]{%
  \def\@volnumber{\uppercase{#1}}}
\AtBeginDocument{%
\def\ps@plain{\ps@empty
  \def\@oddfoot{\@monthyear \hfil \thepage}%
  \def\@evenfoot{\thepage \hfil \@volnumber}}
\def\ps@firstpage{\ps@plain}
\def\ps@headings{\ps@empty
  \def\@evenhead{%
    \setTrue{runhead}%
    \def\thanks{\protect\thanks@warning}%
    \uppercase{\ }\hfil}%
  \def\@oddhead{%
    \setTrue{runhead}%
    \def\thanks{\protect\thanks@warning}%
    \hfill\uppercase{Zeckendorf's Theorem Using Indices in an Arithmetic Progression}}%
  \let\@mkboth\markboth
  \def\@evenfoot{%
    \thepage \hfil \@volnumber}%
  \def\@oddfoot{%
    \@monthyear \hfil \thepage}%
  }%
\footskip=25pt
\pagestyle{headings}%
}
\makeatother

\theoremstyle{plain}
\numberwithin{equation}{section}
\newtheorem{thm}{Theorem}[section]

\newtheorem{lemma}[thm]{Lemma}
\newtheorem{remark}[thm]{Remark}

\newtheorem{example}[thm]{Example}
\newtheorem{definition}[thm]{Definition}

\newcommand{\ignore}[1]{}

\newcommand\be{\begin{eqnarray}}
\newcommand\ee{\end{eqnarray}}
\newcommand\bea{\begin{eqnarray}}
\newcommand\eea{\end{eqnarray}}
\newcommand\ben{\begin{enumerate}}
\newcommand\een{\end{enumerate}}

\begin{document}




\monthyear{TBD}
\volnumber{Volume, Number}
\setcounter{page}{1}

\title{Zeckendorf's Theorem Using Indices in an Arithmetic Progression}

\author{Amelia Gilson, Hadley Killen, Tamás Lengyel, Steven J. Miller, \\ Nadia Razek, Joshua M. Siktar, and Liza Sulkin \\ }

\address{Department of
  Mathematical Sciences, Carnegie Mellon University, Pittsburgh, PA 15213}

\address{\tiny{Department of Mathematical Sciences, Carnegie Mellon University, Pittsburgh, PA 15213}} \email{ameliasg97@gmail.com}

\address{\tiny{Department of Mathematical Sciences, Carnegie Mellon University, Pittsburgh, PA 15213}} \email{hadley.killen@gmail.com}

\address{\tiny{Department of Mathematics, Occidental College, 1600 Campus Road, Los Angeles, CA 90041}} \email{lengyel@oxy.edu}

\address{\tiny{Department of Mathematical Sciences, Carnegie Mellon University, Pittsburgh, PA 15213, and Department of Mathematics and Statistics, Williams College, Williamstown, MA 01267}} \email{sjm1@williams.edu}

\address{\tiny{Department of Mathematical Sciences, Carnegie Mellon University, Pittsburgh, PA 15213}} \email{nadia.razek97@gmail.com}

\address{\tiny{Department of Mathematics, The University of Tennessee-Knoxville, Knoxville, TN 37916}} \email{jsiktar@vols.utk.edu}

\address{\tiny{Department of Mathematical Sciences, Carnegie Mellon University, Pittsburgh, PA 15213}} \email{liza.sulkin@gmail.com}

\thanks{This work was supported in part by NSF Grant DMS1561945. We thank the referee for a careful reading.}

\begin{abstract}
Zeckendorf's Theorem states that any positive integer can be uniquely decomposed into a sum of distinct, non-adjacent Fibonacci numbers. There are many generalizations, including results on existence of decompositions using only even indexed Fibonacci numbers. We extend these further and prove that similar results hold when only using indices in a given arithmetic progression. As part of our proofs, we generate a range of new recurrences for the Fibonacci numbers that are of interest in their own right.
\end{abstract}

\maketitle
\tableofcontents

\section{Introduction}
The Fibonacci sequence is defined via the recurrence relation
\be\label{FibDef}
F_n \ = \ F_{n-1}+F_{n-2}
\ee
for $n\geq 2$, where we need two initial conditions; often these are $F_0=0$ and $F_1=1$. We can use Binet's Formula to jump to the $n$\textsuperscript{th} term:
\be \label{BinetFormula}
F_n \ = \ \frac{1}{\sqrt{5}}\left(\left(\frac{1 + \sqrt{5}}{2}\right)^n - \left(\frac{1 - \sqrt{5}}{2}\right)^n\right) \ = \ \frac{\phi^n-(-\phi)^{-n}}{\sqrt{5}},
\ee where $\phi$ is the Golden ratio $\frac{1 + \sqrt{5}}{2}$.

There are many interesting properties of the Fibonacci numbers; see for example \cite{Kos}. We focus on Zeckendorf's Theorem; it turns out that if we change the initial conditions, the Fibonacci numbers are equivalent to a decomposition property of the integers.\footnote{If we began with $F_0 = 0, F_1 = 1$ then $F_2 = 1$ and we lose uniqueness of decomposition, both because we can add an $F_0$ as well as we have two ways to represent $1$.}

\begin{thm}\label{ZeckThm}(Zeckendorf's Theorem)
\label{zeck} Consider the Fibonacci recurrence with initial conditions $F_0 = 1, F_1 = 2$. Any positive integer $n$ can be expressed uniquely as a sum of non-adjacent Fibonacci numbers: \be\label{zeckDecomp} N\ = \ \sum^{\infty}_{k = 0} b_kF_k, \ \ {\rm where} \ \ b_k \in \{0,1\} \ {\rm and} \ b_k\cdot b_{k+1}\ =\ 0. \ee Further, the Fibonacci numbers are the unique sequence of positive numbers such that every integer can be expressed uniquely as a sum of non-adjacent terms. Note we could also choose initial conditions $F_0 = 0, F_1 = 1$ if we only use indices $k \ge 2$.
\end{thm}

The classical proof is by induction on $N$, but other proofs have been developed as well; see \cite{Br, KKMW, Len1, Ost, Ze}. There is an extensive literature on generalizations and variations of Theorem \ref{ZeckThm}; see for instance \cite{Al, BM, Br, CHHMPV, DDKMMV, Fr, Ho, Ke, Luo, ML, MW1, MW2}. Zeckendorf decompositions have also been studied in a combinatorial framework in numerous places, including \cite{BCCFLMX, CCGJMSY, FJLLLMSS, KKMW, Len1}. The combinatorial approach initiated in \cite{KKMW} is very useful for studying related problems, such as the distribution of the number of summands and the gaps between them in decompositions.








Previous work derived decomposition results when we can only use Fibonacci numbers whose indices have the same parity. For example, there is the \emph{even Fibonacci
representation of $N$} (see \cite{CG1, CG2}): every positive integer has a unique decomposition of the form \be\label{evenDecomp} \sum^{\infty}_{k = 1} b_k F_{2k} \ \ \ {\rm with} \ \ \ b_k \in \{0, 1, 2\} \ \ \ {\rm and\ if\ } b_i = b_j = 2 \ {\rm then}\ \exists k \ {\rm with\ } i < k < j \ {\rm and} \ b_k = 0, \ee
where we use the initial conditions $F_2 = 1, F_3 = 2$ to ensure that the decompositions are unique\footnote{If we were to allow $k = 0,$ then the $F_0 = 0$ term from the Fibonacci Sequence would be allowed in our decompositions and we would lose the uniqueness property of Zeckendorf Decompositions.}. The Fibonacci recurrence decomposes a summand into two terms: one whose index has the same parity as the original summand, and one whose index has the opposite parity. Thus the existence of decomposition \eqref{evenDecomp} is to be expected.

\begin{example} As an example, here are the Zeckendorf and even Fibonacci representations of $83$ respectively:
\begin{align}\label{zeckdecompEx1}
83 \ &= \ 55 + 21 + 5 + 2 \ = \ F_{10} + F_{8} + F_5 + F_3 \nonumber\\
83 \ &= \ 1 \cdot 55 + 1 \cdot 21 + 2 \cdot 3 + 1 \cdot 1 \ = \ 1 \cdot F_{10} + 1 \cdot F_8 + 0 \cdot F_6 + 2 \cdot F_4 + 1 \cdot F_2.
\end{align}
As stated in Theorem \ref{ZeckThm}, the initial conditions for Zeckendorf Decompositions are $F_0 = 0, F_1 = 1$. On the other hand, the even Fibonacci Representation \eqref{evenDecomp} uses the initial conditions $F_2 = 1$ and $F_3 = 2$ to maintain uniqueness of decompositions. For precisely this reason, unlike the decomposition \eqref{zeckDecomp}, we begin summing terms in \eqref{evenDecomp} when $k = 1.$
\end{example}

Given the decomposition \eqref{evenDecomp} result, it is natural to ask whether other subsequences of the Fibonacci numbers also yield unique decompositions, and if so what they are. We prove there are such decompositions when we restrict our indices to be in an arithmetic progression. Before stating our results we first establish some notation.

\begin{definition}($n$-gap Fibonacci numbers)\label{fib n-gap seq}
For $n, m \in \mathbb{N}^+$ with $0 \le m < n+1$, let
$\mathcal{F}(k;n,m) = F_{k(n+1)+m}$ equal the Fibonacci numbers whose indices are congruent to $m$ modulo $n+1$. We call $m$ the offset, and call $\mathcal{F}(k;n,m)$ an $n$-gap subsequence. Note the Fibonacci numbers are a $0$-gap subsequence, and the even and odd index results concern 1-gap subsequences.
\end{definition}

As we will see, the construction of $n$-gap Fibonacci subsequences is based on the theory of \textbf{Positive Linear Recurrence Sequences} (PLRS).

\begin{definition}\label{PLRSdef} (PLRS)
A Positive Linear Recurrence Sequence is a sequence of integers $\{H_n\}_{n=1}^{\infty}$ with the following properties.
\begin{enumerate}
    \item There are non-negative integers $L, c_1,\dots,c_L$ such that
    \be \label{PLRSCond1}
    H_{n} \ = \ c_1H_{n-1} + c_2H_{n-2} + \cdots + c_LH_{n-L}
    \ee
    where $L, c_1, c_L>0$.
    \item $H_1 = 1$ and for $1\leq n<L$, we have
    \be \label{PLRSCond2}
    H_{n} \ = \ c_1H_{n-1} + c_2H_{n-2} + \cdots + c_{n-1}H_1+1.
    \ee
\end{enumerate}
\end{definition}

The study of PLRS is foundational in many papers relating to Zeckendorf Decompositions; see for instance \cite{BCCFLMX, BM, CFHMN, CFHMNPX, DFFHMPP, Ha, ML}. There is an extensive literature on when there is a unique decomposition arising from a given recurrence relation, as well as a host of other properties (such as the distribution of the number of summands in a decomposition, gaps between summands, and digital expansions of these sequences). In particular, if the recurrence relation is a PLRS, then Miller and Wang \cite{MW1, MW2} proved that there exists a unique legal decomposition; for more on these sequences see \cite{BM, Br, Day, DG, Fr, GTNP, Ha, HW, Ho, Ke, KKMW, LT, Len1, PT, Ste1, Ste2}, and for other types of decompositions see \cite{Al, CFHMN, CFHMNPX, CCGJMSY, DDKMMV, DDKMV, DFFHMPP}.



\begin{thm} ($n$-gap Fibonaccis as PLRS)
\label{n-gap PLRSState}
If $n = 2$ or $n \geq 3$ is odd, then the $n$-gap Fibonacci sequence $\{\mathcal{F}(k;n,m)\}^{\infty}_{k = 1}$ is a PLRS for any $0 \leq m < n + 1$. If $n = 2$ then there is a unique decomposition of every positive integer taking the form \eqref{evenDecomp}, with initial conditions $F_2 = 1, F_3 = 2$. On the other hand, if $n \geq 3$ is odd then the decomposition is still unique for every positive integer, but it takes the form
\be\label{nFibDecomp} \sum^{\infty}_{k = 0} b_k F_{k(n + 1) + m} \ \ \ {\rm with} \ \ \ |b_{k}| \ \leq \ a_{n} \ \ \ {\rm for} \ {\rm all} \ \ \ k \ \geq \ 1, \ee
again with initial conditions $F_2 = 1, F_3 = 2$. Here $a_n$ refers to $\phi^n$ rounded to the nearest integer (one of the Lucas numbers).
\end{thm}

Just like the even Fibonacci decomposition, we use the initial conditions $F_2 = 1, F_3 = 2$ for the odd $n$-gap Fibonacci decomposition because we want to ensure all decompositions are unique. Here is an example of what these decompositions look like for specific values of $n$ and $m$.

\begin{example}
This is the $2$-gap decomposition of $143$ when $n = 2, m = 1$:
\be \label{143Ex,n=2,m=1}
143 \ = \ 2 \cdot 55 + 2 \cdot 13 + 2 \cdot 3 + 1 \cdot 1 \ = \ 2 \cdot F_{10} + 2 \cdot F_7 + 2 \cdot F_4 + 1 \cdot F_1.
\ee
Similarly, this is the $2$-gap decomposition of $143$ when $n = 2, m = 2$:
\be \label{143Ex,n=2,m=1=2}
143 \ = \ 1 \cdot 89 + 2 \cdot 21 + 2 \cdot 5 + 2 \cdot 1 \ = \ 1 \cdot F_{11} + 2 \cdot F_8 + 2 \cdot F_5 + 2 \cdot F_2.
\ee
We can also list $3$-gap decompositions for $143$. Here is the decomposition when $n = 3, m = 1$:
\be \label{143Ex,n=3,m=1}
143 \ = \ 4 \cdot 34 + 1 \cdot 5 + 2 \cdot 1 \ = \ 4 \cdot F_9 + 1 \cdot F_5 + 2 \cdot F_1.
\ee
Here is the decomposition when $n = 3, m = 2$:
\be \label{143Ex,n=3,m=2}
143 \ = \ 2 \cdot 55 + 4 \cdot 8 + 1 \cdot 1 \ = \ 2 \cdot F_{10} + 4 \cdot F_6 + 1 \cdot F_2.
\ee
Finally, here is the decomposition when $n = 3, m = 3$:
\be \label{143Ex,n=3,m=3}
143 \ = \ 1 \cdot 89 + 4 \cdot 13 + 1 \cdot 2 \ = \ 1 \cdot F_{11} + 4 \cdot F_7 + 1 \cdot F_3.
\ee
\end{example}


In Section \ref{preliminaries} we examine some recurrences for $n$-gap Fibonacci numbers and discuss how these relate to the more general theory of PLRS. Then in Section \ref{decompResult} we prove Theorem \ref{n-gap PLRSState}. Finally, in Section \ref{conclusion} we give some concluding remarks and possible directions for future research.
\section{Linear Recurrences with Fibonacci Numbers} \label{preliminaries}

We began by looking at decompositions using only every third Fibonacci number; in our notation this would be a $2$-gap Fibonacci sequence with an offset of $2$. We choose this offset so that our first term is $F_2 = 1$, consistent with the initial conditions in \eqref{evenDecomp}. This allows us to begin finding patterns for the general $n$-gap Fibonacci sequence. In this case, we define
\begin{equation}
\label{f(k,2,2)}
\{\mathcal{F}(k;2,2)\}^{\infty}_{k = 0} \ = \ \{1,\ 5,\  21,\  89,\  377,\  1597,\  \dots\}.
\end{equation}
The sequence in \eqref{f(k,2,2)} can itself be defined recursively.

\begin{lemma}
\label{2-gap recurrence}
For $k\geq 2$, \be \label{2gaprec}
F_{3k+2} \ = \ 4\cdot F_{3(k-1)+2}+F_{3(k-2)+2}.
\ee
\end{lemma}

\begin{proof} We repeatedly use the recursion $F_j = F_{j - 1} + F_{j - 2}$ to calculate
\begin{align}\label{2-gap recurrence calc}
F_{3k+2} & \ = \  F_{3k + 1} + F_{3k} \nonumber\\
&\ = \  F_{3k} + F_{3k-1} + F_{3k-1} + F_{3k-2} \nonumber\\
& \ = \  F_{3k-1} + F_{3k-2} + F_{3k-1} + F_{3k-1} + F_{3k-3} + F_{3k-4}\nonumber\\
& \ = \  3\cdot F_{3k-1} + F_{3k-2} + F_{3k-3} + F_{3k-4}\nonumber\\
& \ = \  4\cdot F_{3k-1} + F_{3k-4}\nonumber\\
& \ = \  4\cdot F_{3(k-1)+2}+F_{3(k-2)+2}.
\end{align}
\end{proof}

By a procedure analogous to \eqref{2-gap recurrence calc} we can also generate the following identities, which hold for all $k \geq 2$:
\begin{align}\label{sampleRecurs}
F_{4k+2}& \ = \  7\cdot F_{4(k-1)+2}-F_{4(k-2)+2} \nonumber\\
F_{5k+2}& \ = \  11\cdot F_{5(k-1)+2}+F_{5(k-2)+2} \nonumber\\
F_{6k+2}& \ = \  18\cdot F_{6(k-1)+2}-F_{6(k-2)+2} \nonumber\\
F_{7k+2}& \ = \  29\cdot F_{7(k-1)+2}+F_{7(k-2)+2}.
\end{align}

Notice that $3, 4, 7, 11, 18, 29,\dots$ are the Lucas numbers, which have the closed form $\phi^k+(-\phi)^{-k}$. Since the golden ratio is defined as $\phi = (1 + \sqrt{5})/2$, the Lucas numbers are the closest integer to $\phi^k$ for each $k > 1$, because $\left|(-\phi)^{-k}\right| < 1/2$ for $k > 1$. This motivates the question of whether every $n$-gap Fibonacci subsequence can be defined recursively, and then what decomposition properties they have.  Using Binet's formula \eqref{BinetFormula} for Fibonacci numbers, we can generalize the formulas \eqref{2gaprec} and \eqref{sampleRecurs}.

\begin{lemma}\label{n-gap fib recurrence} For any $n \geq 2$ we have the following generalization of \eqref{2gaprec}:
\be \label{ngaprec}
\mathcal{F}(k;n,m)\ = \ a_n\cdot \mathcal{F}(k-1;n,m)+(-1)^{n-1}\cdot \mathcal{F}(k-1;n,m),
\ee
where $a_n$ henceforth will denote $\phi^n$ rounded to the nearest integer (the Lucas numbers).
\end{lemma}

\begin{proof}
We take advantage of the Lucas numbers and Binet's formula \eqref{BinetFormula} to rewrite each term on the right hand side of \eqref{ngaprec}:
\begin{align}
a_n\cdot \mathcal{F}(k-1;n,m)& \ = \ \left((\phi^n+(-\phi)^{-n})\cdot \frac{1}{\sqrt{5}}(\phi^{nk-n+m}-(-\phi)^{n-nk-m})\right)\nonumber\\
(-1)^{n-1}\cdot \mathcal{F}(k-2;n,m)& \ = \ (-1)^{n-1}\cdot \frac{1}{\sqrt{5}}\left(\phi^{nk-2n+m}-(-\phi)^{-nk+2n-m}\right).
\end{align}

We simplify each component algebraically:
\begin{align}
  a_n\cdot \mathcal{F}(k-1;n,m)& \ = \  \frac{1}{\sqrt{5}}\left((\phi^n+(-\phi)^{-n})(\phi^{nk-n+m}-(-\phi)^{n-nk-m})\right)\nonumber\\
  & \ = \ \frac{1}{\sqrt{5}}\left(\phi^n \cdot \phi^{nk-n+m}-\phi^n\cdot (-\phi)^{n-nk-m}\right.\nonumber\\
  &\quad\ \ + \left.(-\phi)^{-n}\cdot \phi^{nk-n+m}-(-\phi)^{-n}\cdot (-\phi)^{n-nk-m}\right)\nonumber\\
  & \ = \ \frac{1}{\sqrt{5}}\left(\phi^{nk+m}-\phi^n\cdot (-1)^{n-nk-m}\cdot \phi^{n-nk-m}\right.\nonumber\\
  &\quad\ \ + \left.(-1)^{-n}\cdot (\phi)^{-n}\cdot\phi^{nk-n+m}-(-\phi)^{-nk-m}\right) \nonumber\\
  & \ = \ \frac{1}{\sqrt{5}}\left(\phi^{nk+m}+(-1)^{n-nk}\cdot \phi^{2n-nk-m}\right.\nonumber\\
  &\quad\ \ + \left.(-1)^{-n}\cdot \phi^{nk-2n+m}-(-\phi)^{-nk-m}\right),
\end{align}
and
\begin{align}  (-1)^{n-1}\cdot \mathcal{F}(k-2;n,m) & \ = \ \frac{1}{\sqrt{5}}\left((-1)^{n-1}(\phi^{nk-2n+m}-(-\phi)^{2n-nk-m})\right)\nonumber\\
  & \ = \ \frac{1}{\sqrt{5}}\left((-1)^{n-1}\cdot \phi^{nk-2n+m}- (-1)^{n-1}\cdot (-1)^{2n-nk-m}\cdot \phi^{2n-nk-m}\right)\nonumber\\
  & \ = \ \frac{1}{\sqrt{5}}\left((-1)^{n-1}\cdot \phi^{nk-2n+m}- (-1)^{1-n}\cdot (-1)^{2n-nk-m}\cdot \phi^{2n-nk-m}\right)\nonumber\\
  & \ = \ \frac{1}{\sqrt{5}}\left((-1)^{n-1}\cdot \phi^{nk-2n+m}- (-1)^{n-nk}\cdot \phi^{2n-nk-m}\right).
\end{align}

We now sum and simplify the above, and obtain
\begin{align}
&a_n\cdot \mathcal{F}(k-1;n,m)+(-1)^{n-1}\cdot \mathcal{F}(k-1;n,m)\nonumber\\
& \ = \ \frac{1}{\sqrt{5}}\left(\phi^{nk+m}+(-1)^{n-nk}\cdot \phi^{2n-nk-m}+(-1)^{-n}\cdot \phi^{nk-2n+m}-(-\phi)^{-nk-m}\right.\nonumber\\
&\quad \ \ + \ \left.(-1)^{n-1}\cdot \phi^{nk-2n+m}- (-1)^{n-nk}\cdot \phi^{2n-nk-m}\right)\nonumber\\
& \ = \ \frac{1}{\sqrt{5}}\left(\phi^{nk+m}+\cancel{(-1)^{n-nk}\cdot \phi^{2n-nk-m}}+(-1)^{-n}\cdot \phi^{nk-2n+m}-(-\phi)^{-nk-m}\right.\nonumber\\
&\quad \ \ + \ \left.(-1)^{n-1}\cdot \phi^{nk-2n+m}-\cancel{(-1)^{n-nk}\cdot \phi^{2n-nk-m}}\right)\nonumber\\
& \ = \ \frac{1}{\sqrt{5}}\left(\phi^{nk+m}+(-1)^{-n}\cdot \phi^{nk-2n+m}-(-\phi)^{-nk-m}+(-1)^{n-1}\cdot \phi^{nk-2n+m}\right)\nonumber\\
& \ = \ \frac{1}{\sqrt{5}}\left(\phi^{nk+m}+\cancel{(-1)^{-n}\cdot \phi^{nk-2n+m}}-(-\phi)^{-nk-1}+\cancel{(-1)^{n-1}\cdot \phi^{nk-2n+m}}\right)\nonumber\\
& \ = \ \frac{1}{\sqrt{5}}\left(\phi^{nk+m}-(-\phi)^{-nk-m}\right)\nonumber\\
& \ = \ \mathcal{F}(k;n,m),
\end{align} which is the desired result.
\end{proof}

We can generalize Lemma \ref{n-gap fib recurrence} to all $n$-gap subsequences of the recurrence relation
\be \label{G_nRecoriginal}
G_n \ = \ G_{n-1}+G_{n-2},
\ee
where $G_1$ and $G_2$ are positive integers; to do this we first prove a recurrence relating $\{F_{\ell}\}^{\infty}_{\ell = 1}$ to $\{G_{\ell}\}^{\infty}_{\ell = 1}$.

\begin{lemma}
\label{G_nRecIdentity}
If $\{G_{\ell}\}^{\infty}_{\ell = 1}$ satisfies \eqref{G_nRecoriginal} then for $n \geq 3$,
\be \label{G_nRec}
G_n \ = \ F_{n-2}\cdot G_1+F_{n-1}\cdot G_2.
\ee
\end{lemma}

\begin{proof}
We proceed by strong induction. The base case will be $n = 3$, which is verified as follows:
\be\label{G_nRecIdentityEq1}
G_3 \ = \ G_2 + G_1 \ = \ G_2 \cdot 1 + G_1 \cdot 1 \ = \ G_2 \cdot F_2 + G_1 \cdot F_1.
\ee
As for the inductive step, we assume for all $3 \leq j \leq k$ that
\be\label{G_nRecIdentityEq2}
G_j \ = \ F_{j - 2} \cdot G_1 + F_{j - 1} \cdot G_2,
\ee
and we can finish the proof by demonstrating that
\be\label{G_nRecIdentityEq3}
G_{k + 1} \ = \ F_{k - 1} \cdot G_1 + F_k \cdot G_2.
\ee
We can show \eqref{G_nRecIdentityEq3} by using \eqref{G_nRecIdentityEq2} for $j = k$ and $j = k - 1$, along with the recurrence \eqref{FibDef}:
\begin{align}\label{G_nRecIdentityEq4}
G_{k + 1} & \ = \ G_{k} + G_{k - 1} \nonumber \\
& \ = \ F_{k - 2} \cdot G_1 + F_{k - 1} \cdot G_2 + F_{k - 3} \cdot G_1 + F_{k - 2} \cdot G_2 \nonumber \\
& \ = \ (F_{k - 2} + F_{k - 3}) \cdot G_1 + (F_{k - 1} + F_{k - 2}) \cdot G_2 \nonumber \\
& \ = \ F_{k - 1} \cdot G_1 + F_k \cdot G_2,
\end{align}
as desired.
\end{proof}

Now we can prove our generalization of Lemma \ref{n-gap fib recurrence}.

\begin{lemma}
\label{n-gap gen recurrence lemma}
For $k\geq 2$, if $\{G_{\ell}\}^{\infty}_{\ell = 1}$ satisfies \eqref{G_nRecoriginal} then
\be \label{n-gap gen recurrence}
G_{nk+m} \ = \ a_n\cdot G_{n(k-1)+m}+(-1)^{n-1}\cdot G_{n(k-2)+m},
\ee
where $n$, $k$, $m$, and $a_n$ are defined as before, regardless of the initial conditions.
\end{lemma}

\begin{proof}
We can use \eqref{ngaprec} in conjunction with \eqref{G_nRec} to compute
\begin{align}
    a_n\cdot G_{n(k-1)+m}+(-1)^{n-1}\cdot G_{n(k-2)+m}
    & \ = \ a_n\cdot (F_{n(k-1)+m-2}\cdot G_1 +F_{n(k-1)+m-1}\cdot G_2)\nonumber\\
    &\quad\ \ + (-1)^{n-1}(F_{n(k-2)+m-2}\cdot G_1 +F_{n(k-2)+m-1}\cdot G_2)\nonumber\\
    & \ = \ (a_n\cdot F_{n(k-1)+m-2} +(-1)^{n-1}\cdot F_{n(k-2)+m-2})\cdot G_1\nonumber\\
    &\quad\ \ + (a_n\cdot F_{n(k-1)+m-1}+(-1)^{n-1}\cdot F_{n(k-2)+m-1})\cdot G_2\nonumber\\
    & \ = \ F_{nk+m-2}\cdot G_1+F_{nk+m-1}\cdot G_2\nonumber\\
    & \ = \ G_{nk+m}.
\end{align}
\end{proof}

\begin{remark} It turns out there are alternative proofs to Lemmas \ref{n-gap fib recurrence} and \ref{n-gap gen recurrence lemma} that rely on the established generating function theory of multisections . We provide the details in Appendix \ref{multisectionappen}.
\end{remark}

For convenience, we restate the definition of Positive Linear Recurrence Sequences (Definition \ref{PLRSdef}) so that we can solidify the framework to be used in Section \ref{decompResult}.

\ \\

\noindent \textbf{Definition \ref{PLRSdef}.} \emph{A Positive Linear Recurrence Sequence is a sequence of integers $\{H_n\}_{n=1}^{\infty}$ with the following properties.
\begin{enumerate}
    \item There are non-negative integers $L, c_1,\dots,c_L$ such that
    \be \label{PLRSCond1}
    H_{n} \ = \ c_1H_{n-1} + c_2H_{n-2} + \cdots + c_LH_{n-L}
    \ee
    where $L, c_1, c_L > 0$.
    \item $H_1$ = 1 and for $1\leq n<L$, we have:
    \be \label{PLRSCond2}
    H_{n} \ = \ c_1H_{n-1} + c_2H_{n-2} + \cdots + c_{n-1}H_1+1.
    \ee
\end{enumerate}}

\ \\

Here is a simple example demonstrating how to prove that a sequence is a PLRS, particularly the 2-gap sequence.

\begin{example}\label{2gapEx} We can show directly that
\be \{\mathcal{F}(k;2,2)\}^{\infty}_{k = 0} \  = \  \{1, \ 5,\ 21,\ 89,\ 377, \ \dots\}\ee
is a PLRS. Define
\be \{G_k\}_{k=1}^{\infty}\ = \ \{1,\ 5,\ 21,\ 89,\ 377,\ \dots\},\ee
and we check each condition in Definition \ref{PLRSdef}.
\begin{enumerate}
    \item The first condition is true, because we can take $L=2, c_1=4$, and  $c_2=1$; then our recurrence is $G_k = 4G_{k-1}+G_{k-2}$.
    \item The second condition also holds; since $5 = 4\cdot 1+1$, we conclude $G_1 = 1$, and $G_2 = 4G_1+1$.
\end{enumerate}
\end{example}

In the next section we generalize Example \ref{2gapEx}.

\section{Decomposition Results}\label{decompResult}
In this section we restate and prove the main result of the paper, building on the intuition developed in Section \ref{preliminaries}. We quote a primary result from \cite{MW1} that yields the uniqueness of decompositions for the sequences in Section \ref{preliminaries}.

\begin{thm}(Generalized Zeckendorf's Theorem for PLRS)\label{genZeckPLRS}
Let $\{H_j\}_{j=0}^{\infty}$ be a Positive Linear Recurrence Sequence. Then
\begin{enumerate}
\item there is a unique legal decomposition for each positive integer $N \geq 0$, and
\item there is a bijection between the set $S_j$ of integers in $[H_j, H_{j+1})$ and the set $D_j$ of legal decompositions $\sum_{i=1}^j b_i\cdot H_{j+1-i}$.
\end{enumerate}
\end{thm}

We now generalize the result from Example \ref{2gapEx} to the $n$-gap Fibonacci sequences, using Theorem \ref{genZeckPLRS}. Notably, this generalization only extends to odd $n$ due to the $(-1)^{n+1}$ factor of the second term in each recurrence relation amongst the list \eqref{sampleRecurs}. Thus the sequences we study in the next theorem are of the form $\{F_{nk+m}\}^{\infty}_{k=0}$, where $n \geq 3$ is a fixed positive odd integer.

\begin{thm} ($n$-gap Fibonaccis as PLRS)
\label{n-gap PLRS}
If $n \geq 3$ is odd, then the $n$-gap Fibonacci sequence $\{\mathcal{F}(k;n,m)\}^{\infty}_{k = 0}$ is a PLRS.
\end{thm}

\begin{proof} We consider each condition of Definition \ref{PLRSdef} individually:

\begin{enumerate}
    \item The first condition holds via the recurrence relation \eqref{ngaprec}.  Since $F_{nk+1}=a_n\cdot F_{n(k-1)+1}+(-1)^{n-1}\cdot F_{n(k-2)+1}$, we have the following relation for odd $n$:
\be \label{genDecompResultEq1}
F_{nk+1} \ = \ a_n\cdot F_{n(k-1)+1}+F_{n(k-2)+1}.
\ee
Since $a_n$ is a power of a positive number rounded to the nearest integer, we may satisfy the first condition of Definition \ref{PLRSdef} with the following parameters:
\be \label{genDecompResultEq2}
L = 2n,\ \ c_1 = a_n,\ \ c_2 = \cdots = c_{2n-1} = 0, \ \ c_{2n} = 1.
\ee
    \item Since $L=2$, we only need to check the condition \eqref{PLRSCond2} for $m=3$.  For this condition to hold, the following needs to be true for all such sequences of odd $n$:
    \be \label{genDecompResultEq3}
    G_3 \ = \ a_n\cdot G_2+G_1+1.
    \ee
    We can rewrite \eqref{genDecompResultEq3} with Fibonacci numbers as follows:
    \begin{align} \label{genDecompResultEq4}
    F_{2n+1} \ &= \ a_n\cdot F_{n+1}+F_{1}+1 \Rightarrow \nonumber\\
    F_{2n+1}\ &= \ a_n\cdot F_{n+1}+2.
    \end{align}
    The existence of this representation of $G_3$ assures that the proof is complete.    
\end{enumerate}
\end{proof}

Ultimately, this result links together two seemingly unrelated properties: the coefficients of certain PLRS and the necessary conditions for uniqueness of decompositions. We see that $a_n$ acts both as the coefficient of the first term of the $n$-gap Fibonacci recurrence and as the highest coefficient necessary for an integer decomposition using the terms generated by the recurrence.  The sign of the second term of the recurrence in turn determines whether the integer decompositions are unique, and where a positive term corresponds to uniqueness.  This naturally extends to linear combinations of $n$-gap Fibonaccis, i.e. the recurrences of the form
\be \label{linCombGener}
G_k \ = \ G_{k-1}+G_{k-2},
\ee
where $G_1,G_2\in \mathbb{Z}^+$.

\section{Conclusion and Future Work}\label{conclusion}

Our method of looking specifically at $n$-gap Fibonacci sequences has lead us to several generalizations of Zeckendorf's theorem. We were able to connect these problems to the literature on PLRS by concluding that odd gap Fibonacci sequences are PLRS, and by utilizing results on the number of decompositions of natural numbers that exist using the elements of said sequences. The natural open problem to investigate is to determine whether these results can be extended to even integers $n \geq 4$.

Aside from the even integers case, there are also natural enumeration questions to consider. We could study the number of decompositions that arise if we remove the restriction placed by the recursive relationship, but still including the restriction on the number of copies of each summand. Alternatively, we could remove the restriction on the number of copies and investigate how to count the resulting decompositions.

Another possible direction is exploring different types of sequences beyond the Fibonnaci numbers, such as Skiponaccis ($S_k=S_{k-1}+S_{k-3}$), Tribonaccis ($T_k=T_{k-1}+T_{k-2}+T_{k-3}$), and so on, and seeing if we can find potential positive linear recursive sequences by changing the values of the coefficients. We could also explore trying to extend our work to cover sequences of the form $G_n = \alpha G_{n-1}+\beta G_{n-2}$, where $\alpha$ and $\beta$ are arbitrary integers.


\appendix

\section{Alternative Proofs of Lemmas \ref{n-gap fib recurrence} and \ref{n-gap gen recurrence lemma}}\label{multisectionappen}

One can use multisection techniques \cite{GL, Len2} for the generating functions to obtain alternative proofs of Lemmas \ref{n-gap fib recurrence} and \ref{n-gap gen recurrence lemma}. The steps are as follows.\\

The following generating function expansions for Fibonacci Numbers and Lucas Numbers are well-known.
\begin{eqnarray}\label{GFexpansion}
f(x,1,0)&=\sum_{k=0}^\infty F_{k} x^k=\frac{x}{1-x-x^2} \nonumber \\
l(x,1,0)&=\sum_{k=0}^\infty L_{k} x^k=\frac{2-x}{1-x-x^2}.
\end{eqnarray}

Then by invoking the multisection technique on \eqref{GFexpansion} we obtain

\begin{eqnarray}\label{GFexpansionmultisection}
f(x,n,m)&=\sum_{k=0}^\infty F_{kn+m} x^k=\frac{F_m+(F_{n+m}-F_mL_n)x}{1-L_nx+(-1)^n x^2} \nonumber \\
l(x,n,m)&=\sum_{k=0}^\infty L_{kn+m} x^k=\frac{L_m+(L_{n+m}-L_mL_n)x}{1-L_nx+(-1)^n x^2}.
\end{eqnarray}

For the special case $n=1$ and $m=0$ \cite{Len2}, we have that
\be \label{GFmulti_n=1_m=0}
f(x,n,0)=\frac{F_n x}{1-L_n x+(-1)^n x^2}  \text { and } l(x,n,0)=\frac{2-L_nx}{1-L_nx+(-1)^nx^2}.
\ee

The key realization is that the generating extend to negative integer values for the parameters $m$ and $n$.

\medskip

\end{document}